\documentclass{svproc}
\usepackage{url}

\usepackage{amssymb}
\usepackage{graphicx}
\usepackage{footmisc}
\usepackage{multicol}
\usepackage{color,soul}
\usepackage{subeqnarray}
\usepackage[square,numbers]{natbib}
\newcommand{\sr}[1]{{\color{black}{#1}}}

\begin{document}
\mainmatter              
\title{Robust parameter estimation using the ensemble Kalman filter}
\titlerunning{Frequentist perspective on estimation using the EnKF}  
%
\author{Sebastian Reich\inst{1}}
\authorrunning{Sebastian Reich} 
%
\tocauthor{Sebastian Reich}
\institute{University of Potsdam, Institute of Mathematics, Potsdam, Germany\\
\email{sebastian.reich@uni-postdam.de},
}

\maketitle              

\begin{abstract}
Standard maximum likelihood or Bayesian approaches to parameter estimation for stochastic differential equations are not robust to perturbations in the continuous-in-time data. In this paper, we give a rather elementary explanation of this observation in the context of continuous-time parameter estimation using an ensemble Kalman filter. We employ the frequentist perspective to shed new light on three  robust estimation techniques; namely subsampling the data, rough path corrections, and data filtering. 
We illustrate our findings through a simple numerical experiment.
\keywords{parameter estimation, stochastic differential equations, ensemble Kalman filter, frequentist approach, rough path theory}
\end{abstract}
\section{Introduction}
In this note, which is an extended version of \cite{Reich2021}, we consider the well-studied problem of parameter estimation for stochastic differential equations (SDEs) from \sr{continuous-time observations $X_t^\dagger$, $t\in [0,T]$} \cite{kutoyants2013statistical}. \sr{It is well-known that the corresponding maximum likelihood estimator does not depend continuously on the observations $X_t^\dagger$, $t \in [0,T]$, which can result in a systematic estimation bias \cite{PPS09,DFM2016}. In other words, the maximum likelihood estimator is not robust with respect to perturbations in the observations.} Here, we revisit this problem from the perspective of online  (time-continuous) parameter estimation \cite{crisan,crisan2013robust} using the popular ensemble Kalman filter (EnKF) and its continuous-time ensemble Kalman--Bucy filter (EnKBF) formulations \cite{evensen,CotterReich2013,nusken2019state}. \sr{As for the corresponding maximum likelihood approaches, the EnKBF does not depend continuously on the incoming observations $X_t^\dagger$, $t\ge 0$, with respect to the uniform norm topology on the space of continuous functions. This fact has been first  investigated in \cite{CNN2021} using rough path theory \cite{FrizHairer2020}. In particular, as already demonstrated for the related maximum likelihood estimator in \cite{DFM2016}, rough path theory allows one to specify an appropriately generalised topology which leads to a continuous dependence of the EnKBF estimators on the observations. Here we expand the analysis of \cite{CNN2021} to a frequentist analysis of the EnKBF in the spirit of \cite{RR20}, where the primary focus is on the expected behaviour of the EnKBF estimators over all admissible observation paths. One recovers that the discontinuous dependence of the EnKBF estimators on the driving observations results in a systematic bias from a frequentist perspective. This is also a well known fact for SDEs driven by multiplicative noise \cite{IW89}. 

The proposed frequentist perspective naturally enables the study of known bias correction methods, such as subsampling the data \cite{PPS09}, a recently proposed data filtering approach \cite{AGPSZ21}, as well as novel de-biasing approaches \cite{CNN2021} in the context of the EnKBF.}

In order to facilitate a rather elementary mathematical analysis, we consider only the very much simplified problem of parameter estimation for linear SDEs. This restriction allows us to avoid certain technicalities from rough path theory and enables a rather straightforward application of the numerical rough path approach put forward in \cite{davie2008differential}. As a result we are able to demonstrate that the popular approach of subsampling the data \cite{ait2005often,PPS09,azencott2013sub} can be well justified from a frequentist perspective. The frequentist perspective also suggests a rather natural approach to the estimation of the required correction term in the case an EnKBF is implemented without subsampling. 

We end this introductory paragraph with a reference to \cite{abdulle2020drift}, which includes a broad survey on alternative estimation techniques. We also point to \cite{CNN2021} for an in-depth discussion of rough path theory in connection to filtering and parameter estimation. 

The remainder of this paper is structured as follows. The problem setting and the EnKBF are introduced in the subsequent Section 
\ref{sec:EnKF}. The frequentist perspective and its implications on the specific implementations of an EnKBF in the context of low and high frequency data assimilation are laid out in Section \ref{sec:frequentist}. The importance of these considerations becomes transparent when applying the EnKBF to perturbed data in Section \ref{sec:robust}. Here again, we restrict attention to a rather simple model setting taken from \cite{friz2015physical} and also used in \cite{CNN2021}. As a result we build a clear connection between subsampling and the necessity for a correction term in the case high frequency data is assimilated directly. 
We also provide a discussion of the data filtering approach \cite{AGPSZ21} in the context of our simply model system. A brief numerical demonstration is provided in Section \ref{sec:numerics}, which is followed by a concluding remark in Section \ref{sec:conclusions}.

%
\section{Ensemble Kalman parameter estimation} \label{sec:EnKF}
%

We consider the SDE parameter estimation problem
\begin{equation} \label{eq:SDE_1}
{\rm d} X_t = f(X_t,\theta){\rm d}t + \gamma^{1/2} {\rm d}W_t
\end{equation}
subject to observations $X_t^\dagger $, $t\in [0,T]$, which arise from the reference system
\begin{equation} \label{eq:SDE_2}
{\rm d} X_t^\dagger  = f^\dagger (X_t^\dagger){\rm d}t + \gamma^{1/2} {\rm d}W^\dagger_t,
\end{equation}
where the unknown drift function $f^\dagger (x)$ typically satisfies $f^\dagger (x) = f(x,\theta^\dagger)$ and $\theta^\dagger$ 
denotes the true parameter value. Here we assume for simplicity that the unknown parameter is scalar-valued and that the state 
variable is $d$-dimensional with $d\ge 1$. Furthermore, $W_t$ and $W_t^\dagger$ denote independent standard $d$-dimensional Brownian motions and $\gamma>0$ is the (known) diffusion constant.

Following the Bayesian paradigm, we treat the unknown parameter as a random variable $\Theta$. Furthermore, we apply a sequential approach and update $\Theta$ with the incoming data $X^\dagger_t$ as a function of time. \sr{Hence we introduce the
random variable $\Theta_t$ which obeys the Bayesian posterior distribution given all observations $X_\tau^\dagger$, $\tau \in 
[0,t]$, up to time $t>0$. Furthermore, instead of exactly solving the time-continuous Bayesian inference problem as specified by the associated Kushner--Stratonovitch equation \cite{crisan,nusken2019state}, we define the time evolution of $\Theta_t$}
by an application of the (deterministic) ensemble Kalman--Bucy filter (EnKBF) mean-field equations
\cite{CotterReich2013,nusken2019state}, which take the form
\begin{subeqnarray} \label{eq:EnKBF_1}
{\rm d} \Theta_t &=& \gamma^{-1} \pi_t \left[(\theta - \pi_t[\theta]) \otimes f(X_t^\dagger ,\theta) \right] {\rm d} I_t,\\
{\rm d} I_t &=& {\rm d}X_t^\dagger 
- \frac{1}{2} \left( f(X_t^\dagger,\Theta_t ) + \pi_t[f(X^\dagger_t,\theta)] \right) {\rm d}t ,
\end{subeqnarray}
where $\pi_t$ denotes the probability density function (PDF) of $\Theta_t$ and $\pi_t[g]$ the associated expectation value of a function $g(\theta)$. The column vector $I_t$, defined by (\ref{eq:EnKBF_1}b), is called the innovation, 
while the row vector
\begin{equation}
K_t(\pi_t) =  \gamma^{-1} \pi_t \left[(\theta - \pi_t[\theta]) \otimes f(X_t^\dagger ,\theta) \right] ,
\end{equation}
premultiplying the innovation in (\ref{eq:EnKBF_1}a) is called the gain. \sr{Here the notation $a \otimes b = ab^{\rm T}$, where $a,b$ can be any two column vectors, has been used.} The initial condition $\Theta_0 \sim \pi_0$ is provided by the prior PDF of the unknown parameter.

\sr{A Monte-Carlo implementation of the mean-field equations (\ref{eq:EnKBF_1}) leads to the interacting particle system
\begin{subeqnarray} \label{eq:EnKBF_1M}
{\rm d} \Theta_t^{(i)} &=& \gamma^{-1} \pi_t^M \left[(\theta - \pi_t^M[\theta]) \otimes f(X_t^\dagger ,\theta) \right] {\rm d} I_t^{(i)},\\
{\rm d} I_t^{(i)} &=& {\rm d}X_t^\dagger 
- \frac{1}{2} \left( f(X_t^\dagger,\Theta_t^{(i)} ) + \pi_t^M [f(X^\dagger_t,\theta)] \right) {\rm d}t ,
\end{subeqnarray}
$i = 1,\ldots,M$, where expectations are now taken with respect to the empirical measure. That is,
\begin{equation}
\pi_t^M [g] = \frac{1}{M} \sum_{i=1}^M g(\Theta_t^{(i)})
\end{equation}
for given function $g(\theta)$, and all Monte-Carlo samples are driven by the same (fixed) observations $X_t^\dagger$.
The initial samples $\Theta_0^{(i)}$, $i=1,\ldots,M$, are drawn identically and independently from the prior distribution 
$\pi_0$.}

We note in passing that there is also a stochastic variant of the innovation process \cite{nusken2019state} defined by
\begin{equation}
{\rm d} I_t = {\rm d}X_t^\dagger 
- f(X_t^\dagger,\Theta_t )  {\rm d}t - \gamma^{1/2}{\rm d}W_t ,
\end{equation}
\sr{which leads to the Monte-Carlo approximation
\begin{equation}
{\rm d} I_t^{(i)} = {\rm d}X_t^\dagger 
- f(X_t^\dagger,\Theta_t^{(i)} ) {\rm d}t - \gamma^{1/2} {\rm d}W_t^{(i)} 
\end{equation}
of the innovation in (\ref{eq:EnKBF_1M}).}

\begin{remark}
\sr{There is an intriguing connection to the stochastic gradient descent approach to the estimation of $\theta^\dagger$,
as proposed in \cite{SS17}, which is written as
\begin{subeqnarray} \label{eq:SGD}
{\rm d}\theta_t &=& \frac{\alpha_t}{\gamma} \nabla_\theta f(X_t^\dagger,\theta_t){\rm d}\tilde I_t,\\
{\rm d}\tilde I_t &=& {\rm d}X_t^\dagger - f(X_t^\dagger,\theta_t){\rm d}t
\end{subeqnarray}
in our notation, where $\alpha_t >0$ denotes the learning rate. We note that (\ref{eq:SGD}) shares with (\ref{eq:EnKBF_1}) the gain times innovation structure. However, while (\ref{eq:EnKBF_1}) approximates the Bayesian inference problem, formulation (\ref{eq:SGD}) treats the parameter estimation problem from an optimisation perspective. Both formulations share, however, the discontinuous dependence on the observation path $X_t^\dagger$, and the proposed frequentist analysis of the EnKBF 
(\ref{eq:EnKBF_1}) also applies in simplified form to (\ref{eq:SGD}). We also point out that (\ref{eq:EnKBF_1}) is affine invariant \cite{GINR19} and does not require the computation of partial derivatives.}
\end{remark}

\medskip

\noindent
We now state a numerical implementation with step-size $\Delta t>0$ and denote the resulting 
numerical approximations at $t_n = n\Delta t$ by $\Theta_n\sim \pi_n$, $n \ge 1$. \sr{While a standard Euler--Maruyama approximation could be applied, the following stable} discrete-time mean-field formulation of the 
EnKBF
\begin{equation} \label{eq:EnKF_1}
\Theta_{n+1} = \Theta_n + K_n \left\{ (X_{t_{n+1}}^\dagger - X_{t_n}^\dagger)  
- \frac{1}{2} \left( f(X_{t_n}^\dagger,\Theta_n ) + \pi_n[f(X^\dagger_{t_n},\theta)] \right)\Delta t \right\} 
\end{equation}
 is inspired by \cite{akir11} with Kalman gain
\begin{subeqnarray}
K_n &=& \pi_n \left[(\theta - \pi_n[\theta]) \otimes f(X_{t_n}^\dagger ,\theta)\right] \times \\
& &\quad 
\left( \gamma + \Delta t \pi_n \left[ \left(f(X_{t_n}^\dagger ,\theta) -
\pi_n [f(X_{t_n}^\dagger ,\theta)] \right)  \otimes f(X_{t_n}^\dagger ,\theta)\right] \right)^{-1}.
\end{subeqnarray}
\sr{It is straightforward to combine this time discretisation with the Monte-Carlo approximation (\ref{eq:EnKBF_1M}) in order to
obtain a complete numerical implementation of the EnKBF.}

\begin{remark}
The rough path analysis of the EnKBF presented in \cite{CNN2021} is based on a Stratonovich reformulation 
of (\ref{eq:EnKBF_1}) and its appropriate time discretisation. Here we follow the It\^o/Euler--Maruyama formulation of the data-driven term in (\ref{eq:EnKBF_1}), 
\begin{equation} \label{eq:Ito_interpretation}
\int_0^T g(X_t^\dagger,t)\,{\rm d}X_t^\dagger = \lim_{\Delta t\to 0} \sum_{i=1}^L g(X_{t_n}^\dagger,t_n)(X_{t_{n+1}}^\dagger -
X_{t_n}^\dagger )
\end{equation}
for any continuous function $g(x,t)$ and $\Delta t=T/L$, as it corresponds to standard implementation of the EnKBF 
and is easier to analyse in the context of this paper.
\end{remark}

\medskip

\noindent
The EnKBF provides only an approximate solution to the Bayesian inference problem for general nonlinear $f(x,\theta)$.
However, it becomes exact in the mean-field limit for affine drift functions $f(x,\theta) = \theta Ax + Bx + c$. 

\begin{example}
Consider the stochastic partial differential equation
\begin{equation}
\partial_t u = -U\partial_y u  + \rho \partial_y^2 u + \dot{\mathcal{W}}
\end{equation}
over a periodic spatial domain $y \in [0,L)$, where $\mathcal{W}(t,y)$ denotes space-time white noise, $U\in \mathbb{R}$, and $\rho>0$ are given parameters. A standard finite-difference discretisation in space with $d$ grid points and mesh-size $\Delta y$ leads to a linear system of SDEs of the form
\begin{equation}
{\rm d}{\bf u}_t = -(U D +  \rho D D^{\rm T}){\bf u}_t{\rm d}t + \Delta y^{-1/2} {\rm d}W_t,
\end{equation}
where ${\bf u}_t \in \mathbb{R}^d$ denotes the vector of grid approximations at time $t$, $D \in \mathbb{R}^{d\times d}$ 
a finite difference approximation of the spatial derivative $\partial_y$, and $W_t$ the standard $d$-dimensional Brownian motion. We can now set $X_t={\bf u}_t$, $\gamma = \Delta y^{-1}$ and identify either $\theta = U$ or $\theta = \rho$ as the unknown parameter in order to obtain an SDE of the form (\ref{eq:SDE_1}).
\end{example}

\medskip

\noindent
In this note, we further simplify our given inference problem to the case
\begin{equation}
f(x,\theta) = \theta Ax\,,
\end{equation}
where $A \in \mathbb{R}^{d\times d}$ is a normal matrix with eigenvalues in the left half plane. That is $\sigma(A) \subset \mathbb{C}_-$.  The reference parameter value is set to $\theta^\dagger = 1$. Hence the SDE (\ref{eq:SDE_2}) possesses a Gaussian invariant measure with mean zero and covariance matrix 
\begin{equation} \label{eq:C}
C = -\gamma (A + A^{\rm T})^{-1}.
\end{equation}
We assume from now on that the observations $X_t^\dagger$ are realisations of (\ref{eq:SDE_2}) with initial condition $X_0^\dagger \sim {\rm N}(0,C)$.

Under these assumptions, the EnKBF (\ref{eq:EnKBF_1}) simplifies drastically, and we obtain
\begin{subeqnarray} \label{eq:EnKBF_2}
{\rm d} \Theta_t &=&\frac{\sigma_t}{\gamma} (A X^\dagger_t)^{\rm T} {\rm d}I_t,\\
{\rm d} I_t &=& {\rm d}X_t^\dagger 
- \frac{1}{2} \left( \Theta_t + \pi_t[\theta] \right) A X_t^\dagger {\rm d}t ,
\end{subeqnarray}
with variance
\begin{equation}
\sigma_t = \pi_t \left[(\theta - \pi_t[\theta])^2  \right] .
\end{equation}

\begin{remark}
\sr{For completeness, we state the corresponding formulation for the stochastic gradient descent approach
(\ref{eq:SGD}):
\begin{subeqnarray} \label{eq:SGD_2}
{\rm d}\theta_t &=& \frac{\alpha_t}{\gamma} (A X_t^\dagger)^{\rm T} {\rm d}\tilde I_t,\\
{\rm d}\tilde I_t &=& {\rm d}X_t^\dagger - \theta_tA X_t^{\dagger}{\rm d}t.
\end{subeqnarray}
We find that the learning rate $\alpha_t$ takes the role of the variance $\sigma_t$ in (\ref{eq:EnKBF_2}). 
However, we emphasise again that the same pathwise stochastic integrals arise from both formulations, and therefore,
the same robustness issue of the resulting estimators $\theta_t$, $t > 0$, arises.}
\end{remark}

\medskip

\noindent
Similarly, the discrete-time mean-field EnKBF (\ref{eq:EnKF_1}) reduces to
\begin{equation} \label{eq:EnKF_2}
\Theta_{n+1} = \Theta_n + K_n \left\{ (X_{t_{n+1}}^\dagger - X_{t_n}^\dagger)  
- \frac{1}{2} \left( \Theta_n  + \pi_n[\theta] \right) A X_{t_n}^\dagger \Delta t \right\} 
\end{equation}
with Kalman gain
\begin{equation} \label{eq:gain_20}
K_n = \sigma_n (AX_{t_n}^\dagger)^{\rm T} \left( \gamma + \Delta t \sigma_n (AX_{t_n}^\dagger)^{\rm T} AX_{t_n}^\dagger
\right)^{-1}\,.
\end{equation}
Furthermore, since $X_t^\dagger \sim {\rm N}(0,C)$, 
\begin{equation} \label{eq:approx_1}
(AX_t^\dagger)^{\rm T} A X_t^\dagger = (A^{\rm T}A) : (X_t^\dagger \otimes  X_t^\dagger) \approx (A^{\rm T}A):C
\end{equation}
for $d\gg 1$, and we may simplify the Kalman gain to
\begin{equation}
K_n = \sigma_n \,(AX_{t_n}^\dagger)^{\rm T} \left( \gamma + \Delta t \sigma_n  \,(A^{\rm T}A) : C\right)^{-1}.
\end{equation}
\sr{Here we have used the notation $A:B = \mbox{tr} (A^{\rm T} B)$ to denote the Frobenius inner product of two matrices $A,B\in \mathbb{R}^{d\times d}$.} The approximation (\ref{eq:approx_1}) becomes exact in the limit $d\to \infty$, which we will frequently assume in the following section.  \sr{Please note that
\begin{equation} \label{eq:gain_2}
K_n = \frac{\sigma_n}{\gamma} \,(AX_{t_n}^\dagger)^{\rm T} + \mathcal{O}(\Delta t)
\end{equation}
under the stated assumptions.}

\begin{remark}
The Stratonovitch reformulation of (\ref{eq:EnKBF_2}) replaces (\ref{eq:EnKBF_2}a) by
\begin{equation} \label{eq:Strat_formulation}
{\rm d} \Theta_t =\frac{\sigma_t}{\gamma} \left\{ (A X^\dagger_t)^{\rm T} \circ {\rm d}I_t - 
\frac{\gamma}{2} \mbox{tr} \,(A)\,{\rm d}t\right\}.
\end{equation}
The innovation $I_t$ remains as before. See Appendix B of \cite{CNN2021} for more details. An appropriate time discretisation of the innovation-driven term replaces the Kalman gain (\ref{eq:gain_20}) by
\begin{equation}
K_{n+1/2} = \sigma_n (AX_{t_{n+1/2}}^\dagger)^{\rm T} \left( \gamma + \Delta t \sigma_n (AX_{t_{n+1/2}}^\dagger)^{\rm T} AX_{t_{n+1/2}}^\dagger \right)^{-1},
\end{equation}
where
\begin{equation}
X_{t_{n+1/2}}^\dagger = \frac{1}{2} (X_{t_n}^\dagger + X_{t_{n+1}}^\dagger)\,.
\end{equation}
Please note that a midpoint discretisation of the data-driven term in (\ref{eq:Strat_formulation}) results in
\begin{subeqnarray}
(AX_{t_{n+1/2}}^\dagger )^{\rm T}  (X_{t_{n+1}}^\dagger - X_{t_n}^\dagger) &=&
(AX_{t_n}^\dagger )^{\rm T}  (X_{t_{n+1}}^\dagger - X_{t_n}^\dagger) \,\,+ \\
& & \,\,\,\frac{1}{2} A^{\rm T} :
(X_{t_{n+1}}^\dagger - X_{t_n}^\dagger) \otimes (X_{t_{n+1}}^\dagger - X_{t_n}^\dagger)
\end{subeqnarray}
and that
\begin{equation} \label{eq:Strat_corr}
\frac{1}{2} A^{\rm T} :
(X_{t_{n+1}}^\dagger - X_{t_n}^\dagger) \otimes (X_{t_{n+1}}^\dagger - X_{t_n}^\dagger) \approx
\frac{\Delta t \,\gamma}{2} \mbox{tr}\,(A),
\end{equation}
which justifies the additional drift term in (\ref{eq:Strat_formulation}). A precise meaning of the approximation in (\ref{eq:Strat_corr}) will be given in Remark \ref{rem:3} below. 
\end{remark}

\medskip

\noindent
Alternatively, if one wishes to explicitly utilise the availability of continuous-time data $X^\dagger_t$, one could apply the following variant of (\ref{eq:EnKF_2}):
\begin{equation} \label{eq:EnKF_2b}
\Theta_{n+1} = \Theta_n + \frac{\sigma_n}{\gamma} \int_{t_n}^{t_{n+1}} (AX_t^\dagger)^{\rm T} {\rm d} X_t^{\dagger} - \frac{1}{2} K_n A X_{t_n}^\dagger \left( \Theta_n  + \pi_n[\theta] \right) \Delta t ,
\end{equation}
and following the It\^o/Euler--Maruyama approximation (\ref{eq:Ito_interpretation}), discretise the integral with a small inner step-size $\Delta \tau = \Delta t/L$, $L\gg 1$; that is,
\begin{equation} \label{eq:inner_2b}
\int_{t_n}^{t_{n+1}} (AX_t^\dagger)^{\rm T} {\rm d} X_t^{\dagger} \approx
\sum_{l=0}^{L-1} (AX_{\tau_l}^\dagger)^{\rm T} (X_{\tau_{l+1}}^{\dagger}-X_{\tau_l}^\dagger)
\end{equation}
with $\tau_l = t_n + l\Delta \tau$.  We note that
\begin{subeqnarray} \label{eq:rough_path}
\sum_{l=0}^{L-1} (AX_{\tau_l}^\dagger)^{\rm T} (X_{\tau_{l+1}}^{\dagger}-X_{\tau_l}^\dagger) &=&
(AX_{t_n}^\dagger)^{\rm T} (X_{t_{n+1}}^{\dagger}-X_{t_n}^\dagger) \,\,+ \\
& & \,\, A^{\rm T}:
\left( \sum_{l=0}^{L-1} (X_{\tau_l}^\dagger - X_{t_n}^\dagger) \otimes (X_{\tau_{l+1}}^{\dagger}-X_{\tau_l}^\dagger)\right),
\end{subeqnarray}
which is at the heart of rough path analysis \cite{davie2008differential} and which we utilise in the following section.


%
\section{Frequentist analysis} \label{sec:frequentist}
%

\sr{It is well-known that the second-order contribution in (\ref{eq:rough_path}) leads to a discontinuous dependence of the integral
on the observed $X_t^\dagger$ in the uniform norm topology on the space of continuous functions. Rough path theory fixes this problem by defining appropriately extended topologies and has been extended to the EnKBF in \cite{CNN2021}. In this section, we complement the path-wise analysis from \cite{CNN2021} by an analysis of the impact of second-order contribution on the EnKBF (\ref{eq:EnKBF_2}) from a frequentist perspective, which analyses the behaviour of EnKBF over all possible observations $X_t^\dagger$ subject to (\ref{eq:SDE_2}). In other words, one switches from a strong solution concept to a weak one. While we assume that the observations satisfy (\ref{eq:SDE_2}), throughout this section, we will analyse the impact of a perturbed observation process on the EnKBF in Section \ref{sec:robust}.}

\sr{We first derive evolution equations for the conditional mean and variance under the assumption that $\Theta_0$ is Gaussian distributed with given prior mean $m_{\rm prior}$  and variance $\sigma_{\rm prior}$.}  \sr{It follows directly from (\ref{eq:EnKBF_2}) that the conditional mean $\mu_t = \pi_t[\theta]$, that is the mean of $\Theta_t$, satisfies the SDE}
\begin{equation} 
{\rm d}\mu_t = \frac{\sigma_t}{\gamma} \left( (A X_t^\dagger)^{\rm T} {\rm d}X^\dagger_t - \mu_t \,(A^{\rm T}A) : (X_t^\dagger \otimes X_t^\dagger) \,{\rm d}t\right),
\end{equation}
which simplifies to 
\begin{equation} \label{eq:CM_1}
{\rm d}\mu_t = \frac{\sigma_t}{\gamma} \left( (A X_t^\dagger)^{\rm T} {\rm d}X^\dagger_t - \mu_t \,(A^{\rm T}A) : C \,{\rm d}t\right),
\end{equation}
under the approximation (\ref{eq:approx_1}). The initial condition is $\mu_0 = m_{\rm prior}$. 
The evolution equation for the conditional variance, \sr{that is the variance of $\Theta_t$}, is given by
\begin{equation}
\frac{\rm d}{{\rm d}t} \sigma_t = - \frac{\sigma_t^2}{\gamma} \,(A^{\rm T}A):(X_t^\dagger \otimes X_t^\dagger)
\end{equation}
with initial condition $\sigma_0 = \sigma_{\rm prior}$ and which again reduces to 
\begin{equation} \label{eq:variance_1}
\frac{\rm d}{{\rm d}t} \sigma_t = - \frac{\sigma_t^2}{\gamma} \,(A^{\rm T}A):C 
\end{equation}
under the approximation (\ref{eq:approx_1}).

We now perform a frequentist analysis of the estimator $\mu_t$ defined by (\ref{eq:CM_1}) and (\ref{eq:variance_1}), \sr{that is,
we perform a weak analysis of the SDE (\ref{eq:CM_1}) in terms of the first two moments of $\mu_t$ \cite{RR20}.} In the first step, we take the expectation of (\ref{eq:CM_1}) over all realisations $X_t^\dagger$ of the SDE (\ref{eq:SDE_2}), which we denote by 
\begin{equation}
m_t :=\mathbb{E}^\dagger[\mu_t] .
\end{equation} 
The associated evolution equation is given by
\begin{equation}
\frac{\rm d}{{\rm d}t} m_t = 
\frac{\sigma_t}{\gamma} \,(A^{\rm T} A): \mathbb{E}^\dagger \left[X_t^\dagger \otimes X_t^\dagger\right]
-\frac{\sigma_t}{\gamma} \,(A^{\rm T} A) : C \, m_t   ,
\end{equation}
which reduces to
\begin{equation} \label{eq:mean_1}
\frac{\rm d}{{\rm d}t} m_t = \frac{\sigma_t}{\gamma} \,(A^{\rm T}A):C \,(1- m_t) =
\sigma_t \,(A^{\rm T}A): (A + A^{\rm T})^{-1}\,(1-m_t) .
\end{equation}

In the second step, we also look at the frequentist variance
\begin{equation} \label{eq:frequentist_UQ_1}
p_t := \mathbb{E}^\dagger [(\mu_t-m_t)^2] .
\end{equation}
Using
\begin{subeqnarray}
{\rm d}(\mu_t-m_t) &=& \frac{\sigma_t}{\gamma} \left\{
(A^{\rm T}A): \left( X_t^\dagger \otimes X_t^\dagger - C \right){\rm d}t + \gamma^{1/2}
(AX_t^\dagger)^{\rm T} {\rm d}W^\dagger_t \right\} \,\,-\\
& & \qquad \qquad \frac{\sigma_t}{\gamma} (A^{\rm T} A): C \,(\mu_t-m_t){\rm d}t ,
\end{subeqnarray}
we obtain
\begin{subeqnarray}
\frac{\rm d}{{\rm d}t} p_t &=& -\frac{\sigma_t}{\gamma} \,(A^{\rm T}A):C \left(2p_t-\sigma_t\right)\,\,+\\
& &\qquad \qquad \frac{2\sigma_t}{\gamma} \,(A^{\rm T} A): \mathbb{E}^\dagger \left[(X_t^\dagger \otimes X_t^\dagger-C) 
\,(\mu_t-m_t)\right] ,
\end{subeqnarray}
which we simplify to
\begin{equation} \label{eq:frequentist_UQ_1}
\frac{\rm d}{{\rm d}t} p_t = \frac{\sigma_t}{\gamma} \,(A^{\rm T}A):C \left(\sigma_t - 2p_t \right) =
\sigma_t \,(A^{\rm T}A): (A + A^{\rm T})^{-1} \left(\sigma_t - 2p_t\right) 
\end{equation}
under the approximation (\ref{eq:approx_1}). The initial conditions are $m_0 = m_{\rm prior}$ and $p_0 = 0$, respectively. We note that the differential equations (\ref{eq:variance_1}) and (\ref{eq:frequentist_UQ_1}) are explicitly solvable. For example, it holds that
\begin{equation} \label{eq:sigma_exact}
\sigma_t = \frac{\sigma_0}{1 + (A^{\rm T}A) : (A^{\rm T} + A)^{-1}\, \sigma_0 t}
\end{equation}
and one finds that $\sigma_t \sim 1 /((A^{\rm T}A) :(A^{\rm T}+A)^{-1}\,t)$ for $t \gg 1$. It can also be shown that $p_t \le \sigma_t$ for all $t \ge 0$. \sr{Furthermore, this analysis suggests that the learning rate in the stochastic gradient descent formulation (\ref{eq:SGD_2}) should be chosen as
\begin{equation}
\alpha_t = \min \left\{ \bar \alpha, \frac{1}{(A^{\rm T}A) :(A^{\rm T}+A)^{-1}\,t} \right\} ,
\end{equation}
where $\bar \alpha >0$ denotes an initial learning rate; for example $\bar \alpha = \sigma_0$.}

We finally conduct a formal analysis of the ensemble Kalman filter time-stepping (\ref{eq:EnKF_2}) and demonstrate that the method is first-order accurate with regard to the implied frequentist mean $m_t$. We recall (\ref{eq:gain_2})
and conclude from (\ref{eq:EnKF_2}) that the implied update on the variance $\sigma_n$ satisfies
\begin{equation} \label{eq:var_update}
\sigma_{n+1} = \sigma_n - \frac{\sigma_n^2}{\gamma} \,(A^{\rm T} A):C \Delta t + \mathcal{O}(\Delta t^2) ,
\end{equation}
which provides a first-order approximation to (\ref{eq:variance_1}). 

We next analyse the evolution equation (\ref{eq:CM_1}) for the conditional mean $\mu_t$ and its numerical approximation
\begin{equation} \label{eq:CM_2}
\mu_{n+1} = \mu_n + K_n \left\{ (X_{t_{n+1}}^\dagger -X_{t_n}^\dagger) - \mu_n AX_{t_n}^\dagger \Delta t\right\}
\end{equation}
arising from (\ref{eq:EnKF_2}). Here we follow \cite{davie2008differential} in order to analyse the impact of the data $X_t^\dagger$ on the estimator. An in-depth theoretical treatment can be found in \cite{CNN2021}.

Comparing (\ref{eq:CM_2}) to (\ref{eq:CM_1}) and utilising (\ref{eq:gain_2}), we find that the key quantity of 
interest is
\begin{equation} \label{eq:integral}
J^\dagger_{t_n,t_{n+1}} := \int_{t_n}^{t_{n+1}} (AX_t^\dagger)^{\rm T} {\rm d}X_t^\dagger ,
\end{equation}
which we can rewrite as
\begin{equation} 
J^\dagger_{t_n,t_{n+1}} = A^{\rm T}: (X^\dagger_{t_n} \otimes X^\dagger_{t_n,t_{n+1}}) + A^{\rm T} : \mathbb{X}_{t_n,t_{n+1}}^\dagger\,.
\end{equation}
Here, motivated by (\ref{eq:rough_path}) and following standard rough path notation, we have used
\begin{equation}
X^\dagger_{t_n,t_{n+1}} := X_{t_{n+1}}^\dagger -X_{t_n}^\dagger 
\end{equation}
and the second-order iterated It\^o integral
\begin{equation} \label{eq:SE_I}
\mathbb{X}_{t_n,t_{n+1}}^\dagger := \int_{t_n}^{t_{n+1}} (X^\dagger_t - X^\dagger_{t_n})\otimes {\rm d}X_t^\dagger .
\end{equation}
The difference between the integral (\ref{eq:integral}) and its corresponding approximation in (\ref{eq:CM_2}) is provided by
$A^{\rm T} : \mathbb{X}_{t_n,t_{n+1}}^\dagger$ plus higher-order terms arising from (\ref{eq:gain_2}). The iterated integral $\mathbb{X}^\dagger_{t_n,t_{n+1}}$ becomes a random variable from the frequentist perspective. Taking note of (\ref{eq:SDE_2}), we find that the drift, $f(x) = Ax$, contributes with terms of order $\mathcal{O}(\Delta t^2)$ to $\mathbb{X}^\dagger_{t_n,t_{n+1}}$ and the expected value of $\mathbb{X}^\dagger_{t_n,t_{n+1}}$ therefore satisfies
\begin{equation} \label{eq:estimate_4}
\mathbb{E}^\dagger [\mathbb{X}^\dagger_{t_n,t_{n+1}}] =  \mathcal{O}(\Delta t^2) ,
\end{equation}
since $\mathbb{E}^\dagger [W^\dagger_{t_n,\tau}]= 0$ for $\tau > t_n$, and
\begin{equation}
\mathbb{E}^\dagger [\mathbb{W}_{t_n,t_{n+1}}^\dagger] = \frac{1}{2} \mathbb{E}^\dagger [
W^\dagger_{t_n,t_{n+1}} \otimes W^\dagger_{t_n,t_{n+1}} - [W^\dagger_{t_n},W^\dagger_{t_n,t_{n+1}}]  ] 
- \frac{\Delta t}{2}I = 0 ,
\end{equation}
where we have introduced the commutator
\begin{equation} \label{eq:commutator}
 [W_{t_n}^\dagger, W_{t_n,t_{n+1}}^\dagger] := W^\dagger_{t_n}\otimes W^\dagger_{t_n,t_{n+1}} -
W^\dagger_{t_n,t_{n+1}}\otimes W_{t_n}^\dagger .
\end{equation}
Hence we find that, while (\ref{eq:CM_2}) is not a first-order (strong) approximation of the SDE (\ref{eq:CM_1}), the approximation becomes first-order in $m_t$ when averaged over realisations $X_t^\dagger$ of the SDE (\ref{eq:SDE_2}). More precisely, one obtains
\begin{equation}
\mathbb{E}^\dagger [J^\dagger_{t_n,t_{n+1}}] = (A^{\rm T}A) : C \Delta t + \mathcal{O}(\Delta t^2) .
\end{equation}

We note that the modified scheme (\ref{eq:EnKF_2b}) leads to the same time evolution in the variance $\sigma_n$ while the update in $\mu_n$ is changed to
\begin{equation} \label{eq:CM_2b}
\mu_{n+1} = \mu_n +  \frac{\sigma_n}{\gamma} 
\int_{t_n}^{t_{n+1}} (AX_t^\dagger)^{\rm T} {\rm d} X_t^{\dagger} - K_n A X_{t_n}^\dagger \mu_n \Delta t .
\end{equation}
This modification results in a more accurate evolution in the conditional mean $\mu_n$, but because of (\ref{eq:estimate_4}) 
it does not impact to leading order the evolution of the underlying frequentist mean, $m_n = \mathbb{E}^\dagger [\mu_n]$. 
We summarise our findings in the following proposition.

\begin{proposition}
\sr{The discrete-time EnKBF implementations (\ref{eq:EnKF_2}) and (\ref{eq:EnKF_2b}) both provide first-order approximations
to the time evolution of the frequentist mean, $m_t$, and the frequentist variance, $p_t$. In other words, both methods converge weakly with order one.}
\end{proposition}

\medskip

\noindent
We also note that the frequentist uncertainty is essentially data-independent and depends only on the time window $[0,T]$ over which the data gets observed. Hence, for fixed observation interval $[0,T]$, it makes sense to choose the step-size $\Delta t$ such that the discretisation error (bias) remains on the same order of magnitude as $p_T^{1/2} \approx \sigma_T^{1/2}$. Selecting a much smaller step-size would not significantly reduce the frequentist estimation error in the conditional estimator $\mu_T$.  

\begin{remark} \label{rem:3}
We can now give a precise reformulation of the approximation (\ref{eq:Strat_corr}):
\begin{equation}
\frac{1}{2} \mathbb{E}^\dagger \left[ A^{\rm T}: (X_{t_n,t_{n+1}}^\dagger \otimes X_{t_n,t_{n+1}}^\dagger )\right] 
= \frac{\Delta t \,\gamma}{2} \mbox{tr}\,(A) + \mathcal{O}(\Delta t^2),
\end{equation}
which is at the heart of the Stratonovich formulation (\ref{eq:Strat_formulation}) of the EnKFB \cite{CNN2021}.
\end{remark}


%
\section{Multi-scale data} \label{sec:robust}
%

\sr{We now have all the material in place to study the dependency of the EnKBF estimator on a set of observations 
$X_t^{(\epsilon)}$, $\epsilon >0$, which approach the theoretical $X_t^\dagger$ with respect to the uniform norm topology 
on the space of continuous functions as $\epsilon \to 0$. Since the second-order contribution in (\ref{eq:rough_path}), that is 
(\ref{eq:SE_I}), does not depend continuously on such perturbations, we demonstrate in this section that a systematic 
bias arises in the EnKBF. Furthermore, we show how the bias can be eliminated either via subsampling the data, 
which effectively amounts to ignoring these second-order contributions, or via an appropriate correction term, which ensures a continuous dependence on observations $X_t^{(\epsilon)}$ with respect to the uniform norm topology.} More specifically, we investigate the impact of a possible discrepancy between the SDE model (\ref{eq:SDE_1}), for which we aim to estimate the parameter $\theta$, and the data generating SDE (\ref{eq:SDE_2}). We therefore replace (\ref{eq:SDE_2}) by the the following two-scale SDE \cite{friz2015physical}:
\begin{subeqnarray}\label{eq:SDE_4}
{\rm d}X^{(\epsilon)}_t&=& AX^{(\epsilon)}_t\,{\rm d}t + \frac{\gamma^{1/2}}{\epsilon} M P^{(\epsilon)}_t\,{\rm d}t,\\
{\rm d}P^{(\epsilon)}_t&=& -\frac{1}{\epsilon} M P^{(\epsilon)}_t\,{\rm d}t + {\rm d}W^\dagger_t,
\end{subeqnarray}
where
\begin{equation} \label{eq:matrix_M}
M = \left( \begin{array}{cc} 1 & \beta \\ -\beta & 1\end{array} \right),
\end{equation}
$\beta = 2$ and $\epsilon = 0.01$. The dimension of state space is $d=2$ throughout this section. While we restrict here to the simple two-scale model (\ref{eq:SDE_4}), similar scenarios can arise from deterministic fast-slow systems \cite{KellyMelbourne2015,BM18}.

The associated EnKBF mean-field equations in the parameter $\Theta_t$, which we now denote by $\Theta_t^{(\epsilon)}$ in order to explicitly record its dependence on the scale parameter $\epsilon\ll 1$, become
\begin{subeqnarray} \label{eq:EnKBF_3}
{\rm d} \Theta_t^{(\epsilon)} &=&\frac{\sigma_t^{(\epsilon)}}{\gamma} (A X^{(\epsilon)}_t)^{\rm T} {\rm d}I_t^{(\epsilon)},\\
{\rm d} I_t^{(\epsilon)} &=& {\rm d}X_t^{(\epsilon)} 
- \frac{1}{2} \left( \Theta_t^{(\epsilon)} + \pi_t^{(\epsilon)} [\theta] \right) A X_t^{(\epsilon)} {\rm d}t ,
\end{subeqnarray}
with variance
\begin{equation}
\sigma_t^{(\epsilon)} = \pi_t^{(\epsilon)} \left[(\theta - \pi_t^{(\epsilon)} [\theta])^2  \right] 
\end{equation}
and $\Theta^\epsilon_t \sim \pi_t^{(\epsilon)}$. The discrete-time mean-field EnKBF (\ref{eq:EnKF_2}) turns into
\begin{equation} \label{eq:EnKF_3}
\Theta_{n+1}^{(\epsilon)} = \Theta_n^{(\epsilon)} + K_n^{(\epsilon)} \left\{ \left(X_{t_{n+1}}^{(\epsilon)} - X_{t_n}^{(\epsilon)}
\right) - \frac{1}{2} \left( \Theta_n^{(\epsilon)}  + \pi_n^{(\epsilon)}[\theta] \right) A X_{t_n}^{(\epsilon)} \Delta t \right\} 
\end{equation}
with Kalman gain
\begin{equation}
K_n^{(\epsilon)} = \sigma_n^{(\epsilon)} (AX_{t_n}^{(\epsilon)})^{\rm T} \left( \gamma + \Delta t \sigma_n^{(\epsilon)} 
(AX_{t_n}^{(\epsilon)})^{\rm T} AX_{t_n}^{(\epsilon)} \right)^{-1}\,.
\end{equation}
We also consider the appropriately modified scheme (\ref{eq:EnKF_2b}):
\begin{equation} \label{eq:EnKF_3b}
\Theta_{n+1}^{(\epsilon)} = \Theta_n^{(\epsilon)} + \frac{\sigma_n^{(\epsilon)}}{\gamma} 
\int_{t_n}^{t_{n+1}} (AX_t^{(\epsilon)})^{\rm T} 
{\rm d} X_t^{(\epsilon)} - \frac{1}{2} K_n^{(\epsilon)} A X_{t_n}^{(\epsilon)} 
\left( \Theta_n^{(\epsilon)}  + \pi_n^{(\epsilon)}[\theta] \right) \Delta t.
\end{equation}

In order to understand the impact of the modified data generating process on the two mean-field EnKBF formulations 
(\ref{eq:EnKF_3}) and (\ref{eq:EnKF_3b}), respectively, we follow \cite{friz2015physical} and investigate the difference between 
$X^{(\epsilon)}_t$ and $X^\dagger_t$:
\begin{subeqnarray} \label{eq:deviations}
{\rm d} (X^{(\epsilon)}_t - X^\dagger_t) &=&  A (X^{(\epsilon)}_t - X^\dagger_t){\rm d}t + \frac{\gamma^{1/2}}{\epsilon}M P_t^{(\epsilon)} {\rm d}t
-\gamma^{1/2} {\rm d}W_t^\dagger\\
&=&  A (X_t^{(\epsilon)} - X_t^\dagger){\rm d}t - \gamma^{1/2}{\rm d}P_t^{(\epsilon)} .
\end{subeqnarray}
When $P^{(\epsilon)}_t$ is stationary, it is Gaussian with mean zero and covariance 
\begin{equation}
\mathbb{E}_{\rm stat} \left[P_t^{(\epsilon)} \otimes P_t^{(\epsilon)}\right] =\epsilon \,(M + M^{\rm T})^{-1} = \frac{\epsilon}{2} I .
\end{equation}
Hence $P^{(\epsilon)}_t \rightarrow 0$ as $\epsilon \rightarrow 0$ and also
\begin{equation} \label{eq:convergence1}
X^{(\epsilon)}_t \rightarrow X^\dagger_t
\end{equation}
in $L^2$ uniformly in $t$, provided $\sigma(A)\subset \mathbb{C}_-$ and 
$X^{(\epsilon)}_0 = X^{\dagger}_0$. This is illustrated in Figure \ref{fig1}.

\begin{figure}[!htb]
	\begin{center}
	\includegraphics[width=0.9\textwidth]{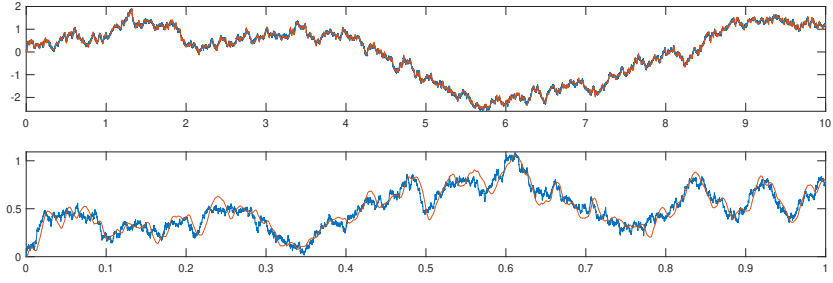}
	\end{center}
	\caption{SDE driven by mathematical vs. physical Brownian motion ($\epsilon = 0.01$). The top panel displays both 
	$X_t^\dagger$ (blue) and $X_t^{(\epsilon)}$ (red) over the long time interval $t\in [0,10]$, while the lower panel provides a zoomed in perspective over the interval $t\in [0,1]$.} \label{fig1}
\end{figure}

In order to investigate the problem further, we study the integral
\begin{equation} \label{eq:integral_eps}
J^{(\epsilon)}_{t_n,t_{n+1}} := \int_{t_n}^{t_{n+1}} (AX_t^{(\epsilon)})^{\rm T} {\rm d}X_t^{(\epsilon)} 
\end{equation}
and its relation to (\ref{eq:integral}). As for (\ref{eq:integral}), we can rewrite (\ref{eq:integral_eps}) as
\begin{equation} 
J^{(\epsilon)}_{t_n,t_{n+1}} = A^{\rm T}: (X^{(\epsilon)}_{t_n} \otimes X^{(\epsilon)}_{t_n,t_{n+1}}) + 
A^{\rm T} : \mathbb{X}_{t_n,t_{n+1}}^{(\epsilon)}.
\end{equation}
We now investigate the limit of the second-order iterated integral 
\begin{subeqnarray} \label{eq:iterated_integral_exact}
\mathbb{X}^{(\epsilon)}_{t_n,t_{n+1}} &=&
 \int_{t_n}^{t_{n+1}} X^{(\epsilon)}_{t_n,t}\otimes {\rm d}X_t^{(\epsilon)}\\
&=&
 \frac{1}{2} X_{t_n,t_{n+1}}^{(\epsilon)} \otimes
X_{t_n,t_{n+1}}^{(\epsilon)} - \frac{1}{2} \int_{t_n}^{t_{n+1}}[ X_{t_n,t}^{(\epsilon)},{\rm d}X_{t}^{(\epsilon)} ]
\end{subeqnarray}
as $\epsilon \to 0$ \cite{friz2015physical}. Here $[.,.]$ denotes the commutator defined by
(\ref{eq:commutator}). 

\begin{proposition} \label{prop2}
\sr{The second-order iterated integral $\mathbb{X}^{(\epsilon)}_{t_n,t_{n+1}}$ satisfies
\begin{equation} 
\lim_{\epsilon \to 0} \mathbb{X}^{(\epsilon)}_{t_n,t_{n+1}} = \mathbb{X}^\dagger_{t_n,t_{n+1}}+ \frac{\Delta t\,\gamma}{2} M
\end{equation}}
\end{proposition}

\begin{proof}
\sr{The proof follows \cite{friz2015physical} and can be summarised as follows:
\begin{subeqnarray} \label{eq:second_order_app}
\mathbb{X}^{(\epsilon)}_{t_n,t_{n+1}} &=& \int_{t_n}^{t_{n+1}} X^{(\epsilon)}_{t_n,t}\otimes {\rm d}X^{(\epsilon)}_t \\
&\rightarrow& \int_{t_n}^{t_{n+1}} X^\dagger_{t_n,t} \otimes  {\rm d} X^\dagger_t-
\gamma^{1/2} \int_{t_n}^{t_{n+1}} X^{(\epsilon)}_{t_n,t}  \otimes {\rm d} P^{(\epsilon)}_t\\
&=& \mathbb{X}^\dagger_{t_n,t_{n+1}} - \gamma^{1/2} X^{(\epsilon)}_{t_n,t_{n+1}} \otimes P^{(\epsilon)}_{t_{n+1}} + 
\gamma^{1/2} \int_{t_n}^{t_{n+1}} {\rm d}X^{(\epsilon)}_t \otimes P^{(\epsilon)}_t \\
& \rightarrow& \mathbb{X}^\dagger_{t_n,t_{n+1}} + \gamma^{1/2} \int_{t_n}^{t_{n+1}} \left\{ A X_t^{(\epsilon)} +
\frac{\gamma^{1/2}}{\epsilon} M P^{(\epsilon)}_t \right\} \otimes  P^{(\epsilon)}_t {\rm d}t \\ 
& \rightarrow& \mathbb{X}^\dagger_{t_n,t_{n+1}} + \frac{\Delta t\,\gamma }{\epsilon}M \,\mathbb{E}_{\rm stat}\left[ P_{t_n}^{(\epsilon)} \otimes  P_{t_n}^{(\epsilon)} \right] \\
&=& \mathbb{X}^\dagger_{t_n,t_{n+1}}+ \frac{\Delta t\,\gamma}{2} M.
\end{subeqnarray}}
\end{proof}

\noindent
As discussed in detail in \cite{CNN2021} already, Proposition \ref{prop2} implies that the scheme (\ref{eq:EnKF_3b}) 
does not, in general, converge to the scheme (\ref{eq:EnKF_3b}) as $\epsilon \to 0$ since
\begin{equation} \label{eq:bias}
J^\dagger_{t_n,t_{n+1}} = \lim_{\epsilon \to 0} J^{(\epsilon)}_{t_n,t_{n+1}} - \frac{\Delta t\,\gamma}{2} A^{\rm T} :M\,.
\end{equation}
This observation suggests the following modification
\begin{subeqnarray} \label{eq:EnKF_3c}
\Theta_{n+1}^{(\epsilon)} &=& \Theta_n^{(\epsilon)} + \frac{\sigma_n^{(\epsilon)}}{\gamma} 
\int_{t_n}^{t_{n+1}} (AX_t^{(\epsilon)})^{\rm T} 
{\rm d} X_t^{(\epsilon)} - \frac{\Delta t}{2} \sigma_n^{(\epsilon)} \,A^{\rm T} : M\,\,-\\
&& \qquad \qquad \frac{1}{2} K_n^{(\epsilon)} A X_{t_n}^{(\epsilon)} 
\left( \Theta_n^{(\epsilon)}  + \pi_n^{(\epsilon)}[\theta] \right) \Delta t
\end{subeqnarray}
to (\ref{eq:EnKF_3b}). 
Please note that it follows from (\ref{eq:iterated_integral_exact}) that
\begin{equation} \label{eq:integral_exact}
\int_{t_n}^{t_{n+1}} (AX_t^{(\epsilon)})^{\rm T} 
{\rm d} X_t^{(\epsilon)} = A^{\rm T} : \left( X_{t_{n+1/2}}^{(\epsilon)} \otimes X_{t_n,t_{n+1}}^{(\epsilon)}
-  \frac{1}{2} \int_{t_n}^{t_{n+1}}[ X_{t_n,t}^{(\epsilon)},{\rm d}X_{t}^{(\epsilon)} ] \right).
\end{equation}

\begin{proposition} 
\sr{The discrete-time EnKBF (\ref{eq:EnKF_3}) converges to (\ref{eq:EnKF_2}) for fixed $\Delta t$ as 
$\epsilon \to 0$.  Similarly, (\ref{eq:EnKF_3c}) converges to (\ref{eq:EnKF_2b}) under the same
limit.}
\end{proposition}

\begin{proof}
\sr{The first statement follows from $\sigma_n^{(\epsilon)} = \sigma_n$, the limiting behaviour (\ref{eq:convergence1}), and
\begin{equation}
\lim_{\epsilon \to 0} K_n^{(\epsilon)} = K_n.
\end{equation}
The second statement additionally requires (\ref{eq:bias}) to be substituted into (\ref{eq:EnKF_3c}) 
when taking the limit $\epsilon \to 0$.}
\end{proof}


\begin{remark}
\sr{The analogous adaptation of (\ref{eq:EnKF_3c}) to the gradient descent formulation 
(\ref{eq:SGD_2}) with $X_t^\dagger$ replaced by $X_t^{(\epsilon)}$ becomes
\begin{subeqnarray}
\theta_{n+1}^{(\epsilon)} &=& \theta_n^{(\epsilon)} + \frac{\alpha_{t_n}}{\gamma} \left(
\int_{t_n}^{t_{n+1}} (AX_t^{(\epsilon)})^{\rm T} {\rm d}X_t^{(\epsilon)} - \frac{\gamma \Delta t}{2} A^{\rm T} : M \,\,- \right.\\
& & \qquad \qquad  \left. \theta_n^{(\epsilon)} (AX_{t_n}^{(\epsilon)})^{\rm T} A X_{t_n}^{(\epsilon)} \Delta t \right).
\end{subeqnarray}
Alternatively, subsampling the data can be applied which leads to the simpler formulation
\begin{equation}
\theta_{n+1}^{(\epsilon)} = \theta_n^{(\epsilon)} + \frac{\alpha_{t_n}}{\gamma} 
(AX_{t_n}^{(\epsilon)})^{\rm T} \left( (X_{t_{n+1}}^{(\epsilon)}-X_{t_n}^{(\epsilon)}) - 
\theta_n^{(\epsilon)}  A X_{t_n}^{(\epsilon)} \Delta t \right).
\end{equation}}
\end{remark}

\begin{remark}
\sr{A two-scale SDE, closely related to (\ref{eq:SDE_4}), has been investigated in \cite{BC13} in terms of 
the time integrated autocorrelation function of $P_t^{(\epsilon)}$ and modified stochastic integrals.  
In our case, the modified quadrature rule, here denoted by $\diamond$, has to satisfy
\begin{equation}
\int_{t_n}^{t_{n+1}} (AX_t^\dagger)^{\rm T} \diamond {\rm d}X_t^\dagger =
\lim_{\epsilon \to 0} \int_{t_n}^{t_{n+1}} (AX_t^{(\epsilon)})^{\rm T} {\rm d}X_t^{(\epsilon)} ,
\end{equation}
and it is therefore related to the standard It\^o integral via
\begin{equation}
\int_{t_n}^{t_{n+1}} (AX_t^\dagger)^{\rm T} \diamond {\rm d}X_t^\dagger =
\int_{t_n}^{t_{n+1}} (AX_t^\dagger)^{\rm T}  {\rm d}X_t^\dagger + \frac{\Delta t \gamma }{2} A^{\rm T} : M.
\end{equation}
Hence $M$ playes the role of the integrated autocorrelation function of $P_t^{(\epsilon)}$ in our approach.
We note that the modified quadrature rule reduces to the standard Stratonovitch integral 
if either $\beta = 0$ in (\ref{eq:matrix_M}) or $A$ is symmetric. 
While the results from \cite{BC13} could, therefore, also be used as a starting point for discussing the 
induced estimation bias, practical implementations would still require knowledge of the integrated autocorrelation function 
of $P_t^{(\epsilon)}$ or, equivalently, the estimation of $M$ in addition to observing $X_t^{(\epsilon)}$. 
We address this aspect next.}
\end{remark}

%
\subsection{Numerical implementation}
%

\sr{The numerical implementation of (\ref{eq:EnKF_3c}) requires an estimator for the generally unknown $M$ in (\ref{eq:bias}). This task is challenging as we only have access to $X_t^{(\epsilon)}$ without any explicit knowledge of the underlying generating process (\ref{eq:SDE_4}).} While the estimator proposed in \cite{CNN2021} is based on the idea of subsampling the data, the frequentist perspective taken in this note suggests the alternative
estimator $M_{\rm est}$ defined by
\begin{equation} \label{eq:M_estimator}
\frac{\Delta t \,\gamma}{2} M_{\rm est} =\mathbb{E}^\dagger [ \mathbb{X}^{(\epsilon)}_{t_n,t_{n+1}}],
\end{equation}
which follows from (\ref{eq:second_order_app}f) and (\ref{eq:estimate_4}). That is, $\mathbb{E}^\dagger [\mathbb{X}^\dagger_{t_n,t_{n+1}}] = \mathcal{O}(\Delta t^2)$ for $\Delta t$ sufficiently small. Note that second-order iterated integral
$X_{t_n,t_{n+1}}^{(\epsilon)}$ satisfies (\ref{eq:iterated_integral_exact}) and is therefore easy to compute.
In practice, the frequentist expectation value can be replaced by an approximation along a given single observation path $X^{(\epsilon)}_t$, $t\in [0,T]$, under the assumption of ergodicity. 

An appropriate choice of the outer or sub-sampling step-size $\Delta t$ \cite{PPS09} constitutes an important aspect for the practical implementation of the EnKBF formulation (\ref{eq:EnKF_3}) for finite values of $\epsilon>0$ \cite{nusken2019state}. Consistency of the second-order iterated integrals \cite{davie2008differential} implies
\begin{equation}
\mathbb{X}^{(\epsilon)}_{t_n,t_{n+2}} = \mathbb{X}^{(\epsilon)}_{t_n,t_{n+1}} + \mathbb{X}^{(\epsilon)}_{t_{n+1},t_{n+2}}
+ X^{(\epsilon)}_{t_n,t_{n+1}} \otimes X^{(\epsilon)}_{t_{n+1},t_{n+2}} .
\end{equation}
A sensible choice of $\Delta t$ is dictated by
\begin{equation}
\mathbb{E}^\dagger \left[ X^{(\epsilon)}_{t_n,t_{n+1}} \otimes X^{(\epsilon)}_{t_{n+1},t_{n+2}}\right] = \mathcal{O}(\Delta t^2) \,,
\end{equation}
that is, the sub-sampled data $X_{t_n}^{(\epsilon)}$ behaves to leading order like solution increments from the reference model (\ref{eq:SDE_2}) at scale $\Delta t$ independent of the specific value of $\epsilon$. Note that, on the other hand,
\begin{equation}
\mathbb{E}^\dagger \left[ X^{(\epsilon)}_{\tau_l,\tau_{l+1}} \otimes X^{(\epsilon)}_{\tau_{l+1},\tau_{l+2}}\right] = \mathcal{O}( \epsilon^{-1} \Delta \tau^2) 
\end{equation}
for an inner step-size $\Delta \tau \sim \epsilon$. In other words, a suitable step-size $\Delta t>0$ can be defined by
making
\begin{equation} \label{eq:subsampling_rate}
h(\Delta t) := \Delta t^{-2} 
\left\|\mathbb{E}^\dagger \left[ X^{(\epsilon)}_{t_n,t_{n+1}} \otimes X^{(\epsilon)}_{t_{n+1},t_{n+2}}\right] \right\|
\end{equation}
as small as possible while still guaranteeing an accurate numerical approximation in (\ref{eq:EnKF_3}).

\begin{remark}
The choice of the outer time step $\Delta t$ is less critical for the EnKBF formulation (\ref{eq:EnKF_3c}) since it does not rely on sub-sampling the data and is robust with regard to perturbations in the data provided the appropriate $M$ is explicitly 
available or has been estimated from the available data using (\ref{eq:M_estimator}). 
Furthermore, if $A$ is symmetric, then it follows from (\ref{eq:integral_exact}) and the skew-symmetry of the commutator
$[.,.]$ that
\begin{equation}
\int_{t_n}^{t_{n+1}} (AX_t^{(\epsilon)})^{\rm T} {\rm d} X_t^{(\epsilon)} 
= A: \left(X^{(\epsilon)}_{t_{n+1/2}}\otimes X^{(\epsilon)}_{t_n,t_{n+1}}\right),
\end{equation}
which can be used in (\ref{eq:EnKF_3c}). The same simplification arises when $M$ is symmetric. This insight is at the heart of the geometric rough path approach followed in \cite{CNN2021} and which starts from the Stratonovich formulation (\ref{eq:Strat_formulation}) of the EnKBF. See also \cite{Pathiraja2020} on the convergence of Wong--Zakai approximations for stochastic differential equations. In all other cases, a more refined numerical approximation of the data-driven integral in (\ref{eq:EnKF_3c}) is necessary; such as, for example, (\ref{eq:inner_2b}). For that reason, we rely on the It\^o/Euler--Maruyama interpretation of (\ref{eq:integral_eps}) in this note instead, that is the approximation (\ref{eq:Ito_interpretation}). 
\end{remark}

%
\subsection{Filtered data} \label{sec:filtered}
%

We finally discuss a recently proposed \cite{AGPSZ21} robust modification to the parameter estimation problem 
in the light of the mean-field EnKBF equations considered in this paper. The essential idea is to filter the observation paths 
$X_t^\dagger$, $t \ge 0$, via 
\begin{equation}
{\rm d}Z_t^\dagger = \frac{1}{\delta} (X_t^\dagger - Z_t^\dagger){\rm d}t + \delta_{\rm noise} \sqrt{2} {\rm d}V_t^\dagger,
\end{equation}
where $\delta >0$ is a sufficiently small parameter and $V_t^\dagger$ denotes 
independent Brownian motion with $\delta_{\rm noise} 
= 1$ (noise added) or $\delta_{\rm noise} = 0$ (no noise added). Extending the methodology proposed in \cite{AGPSZ21} to
the mean-field EnKFB equations (\ref{eq:EnKBF_2}), we now consider
\begin{subeqnarray} \label{eq:EnKBF_2m}
{\rm d} \Theta_t &=&\frac{\sigma_t}{\gamma} (A Z^\dagger_t)^{\rm T} {\rm d}I_t,\\
{\rm d} I_t &=& {\rm d}X_t^\dagger 
- \frac{1}{2} \left( \Theta_t + \pi_t[\theta] \right) A X_t^\dagger {\rm d}t ,
\end{subeqnarray}
with the variance $\sigma_t$ defined as before.

Let us first investigate the long-time behaviour of the extended data generating system
\begin{subeqnarray} \label{eq:extended_obs}
{\rm d}X_t^\dagger &=& AX_t^\dagger {\rm d}t + \gamma^{1/2} {\rm d}W_t^\dagger,\\
{\rm d}Z_t^\dagger &=& \frac{1}{\delta} (X_t^\dagger - Z_t^\dagger){\rm d}t + \delta_{\rm noise} \sqrt{2} {\rm d}V_t^\dagger,
\end{subeqnarray}
in some detail. Its stationary distribution is Gaussian with mean $m_\infty^x = m_\infty^z = 0$. 
The stationary covariance matrices satisfy the relations
\begin{subeqnarray}
0 &=& A\Sigma_\infty^{xx} + \Sigma_\infty^{xx} A^{\rm T} + \gamma I,\\
0 &=& \Sigma_\infty^{xz} + \Sigma_\infty^{zx} - 2( \Sigma_\infty^{zz} - \delta_{\rm noise} \delta I),\\
0 &=& A\Sigma_\infty^{xz} + \frac{1}{\delta} (\Sigma_\infty^{xx} - \Sigma_\infty^{xz}).
\end{subeqnarray}
We note that $\Sigma_\infty^{xx} = C$ with the matrix $C$ defined in (\ref{eq:C}) and 
that the symmetric part of $\Sigma_\infty^{zx}$ and $\Sigma_\infty^{xz}$, respectively,
are equivalent to $\Sigma_\infty^{zz} -  \delta_{\rm noise} \delta I$. Hence, following (\ref{eq:approx_1}), 
we again make the crucial approximation
\begin{equation}
(A^{\rm T}A): Z_t^\dagger \otimes X_t^\dagger \approx (A^{\rm T} A) : \Sigma_\infty^{zx} =
(A^{\rm T} A): (\Sigma_\infty^{zz} - \delta_{\rm noise} \delta I)
\end{equation}
for $d\gg 1$. Let us therefore introduce the shorthand 
\begin{equation} \label{eq:Ctilde}
\tilde C = \Sigma_\infty^{zz} -  \delta_{\rm noise} \delta I. 
\end{equation}
We also note that
\begin{equation}
\tilde C = C + \mathcal{O}(\delta).
\end{equation}

The frequentist analysis from Section \ref{sec:frequentist} delivers
\begin{equation}
{\rm d} \mu_t = \frac{\sigma_t}{\gamma} \left( (AZ_t^\dagger)^{\rm T} {\rm d}X_t^\dagger - \mu_t
(A^{\rm T}A) : \tilde C \,{\rm d}t \right)
\end{equation}
for the conditional mean and
\begin{equation}
\frac{\rm d}{{\rm d}t} \sigma_t = -\frac{\sigma_t^2}{\gamma} (A^{\rm T} A) :\tilde C
\end{equation}
for the conditional variance of the random variable $\Theta_t$, as defined by the modified mean-field evolution equations
(\ref{eq:EnKBF_2m}). Furthermore, we find that $\mu_t$ still provides an asymptotically
unbiased estimator since $m_t = \mathbb{E}^\dagger [\mu_t]$ satisfies
\begin{equation}
\frac{\rm d}{{\rm d}t} m_t = \frac{\sigma_t}{\gamma} (A^{\rm T}A): \tilde C\, (1-m_t).
\end{equation}
Similarly, the variance, $p_t$ of the estimator $\mu_t$ satisfies 
\begin{equation}
\frac{\rm d}{{\rm d}t} p_t = \frac{\sigma_t}{\gamma} (A^{\rm T} A):\tilde C \,(\sigma_t - 2p_t).
\end{equation}
In summary, we find that the modified EnKBF mean-field equations (\ref{eq:EnKBF_2m}) behave exactly as the 
original equations (\ref{eq:EnKBF_2}) with the only difference that the stationary covariance matrix $C = \Sigma_\infty^{xx}$
is replaced everywhere by (\ref{eq:Ctilde}). 

Again following \cite{AGPSZ21}, given multi-scale observations $X_t^{(\epsilon)}$, $t \ge 0$, 
we define associated filtered $Z_t^{(\epsilon)}$ via
\begin{equation}
\frac{\rm d}{{\rm d}t} Z_t^{(\epsilon)} = \frac{1}{\delta} (X_t^{(\epsilon)} - Z_t^{(\epsilon)}) + \delta_{\rm noise} \sqrt{2}{\rm d}V_t^\dagger.
\end{equation}
The intriguing observation is that 
\begin{equation}
\tilde J_{t_n,t_{n+1}}^{(\epsilon)} := \int_{t_n}^{t_{n+1}} (AZ_t^{(\epsilon)})^{\rm T} {\rm d}X_t^{(\epsilon)}
\end{equation}
converges to 
\begin{equation}
\tilde J^\dagger_{t_n,t_{n+1}} := \int_{t_n}^{t_{n+1}} (A Z_t^\dagger)^{\rm T} {\rm d} X_t^\dagger
\end{equation}
as $\epsilon \to 0$. This simply follows from the fact that both integrals can be interpreted as standard 
Riemann--Stieltjes integrals and convergence of $X_t^{(\epsilon)} \to X_t^\dagger$ and 
$Z_t^{(\epsilon)} \to Z_t^\dagger$ as $\epsilon \to 0$ in 
the standard topology of continuous functions is sufficient to conclude the convergence of the associated integrals
$\tilde J_{t_n,t_{n+1}}^{(\epsilon)} \to \tilde J_{t_n,t_{n+1}}^\dagger$.  In other words, the extended EnKBF
formulation
\begin{subeqnarray} \label{eq:EnKBF_2mm}
{\rm d} \Theta_t^{(\epsilon)} &=&\frac{\sigma_t^{(\epsilon)}}{\gamma} (A Z^{(\epsilon)}_t)^{\rm T} {\rm d}I_t^{(\epsilon)},\\
{\rm d} I_t^{(\epsilon)} &=& {\rm d}X_t^{(\epsilon)}
- \frac{1}{2} \left( \Theta_t^{(\epsilon)} + \pi_t^{(\epsilon)}[\theta] \right) A X_t^{(\epsilon)} {\rm d}t ,\\
{\rm d}Z_t^{(\epsilon)} &=& \frac{1}{\delta} (X_t^{(\epsilon)} - Z_t^{(\epsilon)}){\rm d}t + \delta_{\rm noise} \sqrt{2} V_t^\dagger
\end{subeqnarray}
will converge to the correct parameter value $\theta^\dagger = 1$ as $t\to \infty$ in the limit $\epsilon \to 0$, that is,
\begin{equation}
\lim_{t\to \infty} \lim_{\epsilon \to 0} \Theta_t^{(\epsilon)} = \theta^\dagger.
\end{equation}
This statement is in line with the results from \cite{AGPSZ21}. The intriguing point is that the data filtering approach does not require knowledge of the rough path correction term implied by (\ref{eq:bias}) while still delivering an unbiased estimator.

%
\section{Numerical example} \label{sec:numerics}
%

We consider the linear SDE (\ref{eq:SDE_2}) with $\gamma = 1$ and
\begin{equation}
A = \frac{-1}{2} \left( \begin{array}{cc} 1 & -1 \\ 1 & 1 \end{array} \right).
\end{equation}
We find that $C = I$ and $A^{\rm T} A  = 1/2I$. Hence $(A^{\rm T}A ):C = 1$, and the posterior variance simply satisfies $\sigma_t = \sigma_0/(1+\sigma_0 t)$ according to (\ref{eq:sigma_exact}). We set $m_{\rm prior} = 0$ and $\sigma_{\rm prior} = 4$ for the 
Gaussian prior distribution of $\Theta_0$, and the observation interval is $[0,T]$ with $T=6$. We find that $\sigma_T = 0.16$. Solving (\ref{eq:mean_1}) for given $\sigma_t$ with initial condition $m_0 = 0$ yields
\begin{equation} \label{eq:Kalman_prediction}
m_t = 1 - \frac{\sigma_t}{\sigma_0}
\end{equation}
and $m_T = 0.96$. The corresponding curves are displayed in red in Figure \ref{fig2}. 

\begin{figure}[!htb]
	\begin{center}
	\includegraphics[width=0.9\textwidth]{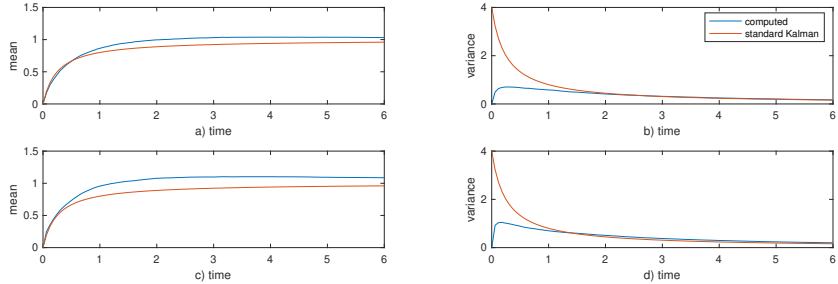}
	\end{center}
	\caption{a)--b): frequentist mean, $m_t$ and variance, $p_t$, from EnKBF implementation (\ref{eq:EnKF_2}) with step-size
	$\Delta t = 0.06$; c)--d): same results from EnKBF implementation (\ref{eq:EnKF_2b}) with inner time-step $\Delta \tau = \Delta t/600$. 
	We also display the curves arising for $\sigma_t$ and $m_t$ from the standard Kalman theory using the approximation (\ref{eq:approx_1}).
	Note that the posterior variance, $\sigma_t$, should provide an upper bound on the frequentist uncertainty $p_t$.} \label{fig2}
\end{figure}

We implement the EnKBF schemes (\ref{eq:EnKF_2}) and (\ref{eq:EnKF_2b}) with $t_n = n\,\Delta t$. 
The inner time-step is $\Delta \tau = 10^{-4}$ while $\Delta t = 0.06$, that is, $L=600$.  We repeat the experiment $N = 10^4$ times and compare the outcome with the predicted mean value of $m_T = 0.96$ and the posterior variance of $\sigma_T = 0.16$ in Figure \ref{fig2}. The differences in the computed time evolutions of $m_t$ and $p_t$ are rather minor and support the idea that it is not necessary to assimilate continuous-time data beyond $\Delta t$. We also find that the simple prediction (\ref{eq:Kalman_prediction}), 
based on standard Kalman filter theory, is not very accurate for this low-dimensional problem ($d=2$). The corresponding 
approximation for $\sigma_t$ provides, however, a good upper bound for $p_t$. 

\begin{figure}[!htb]
	\begin{center}
	\includegraphics[width=0.9\textwidth]{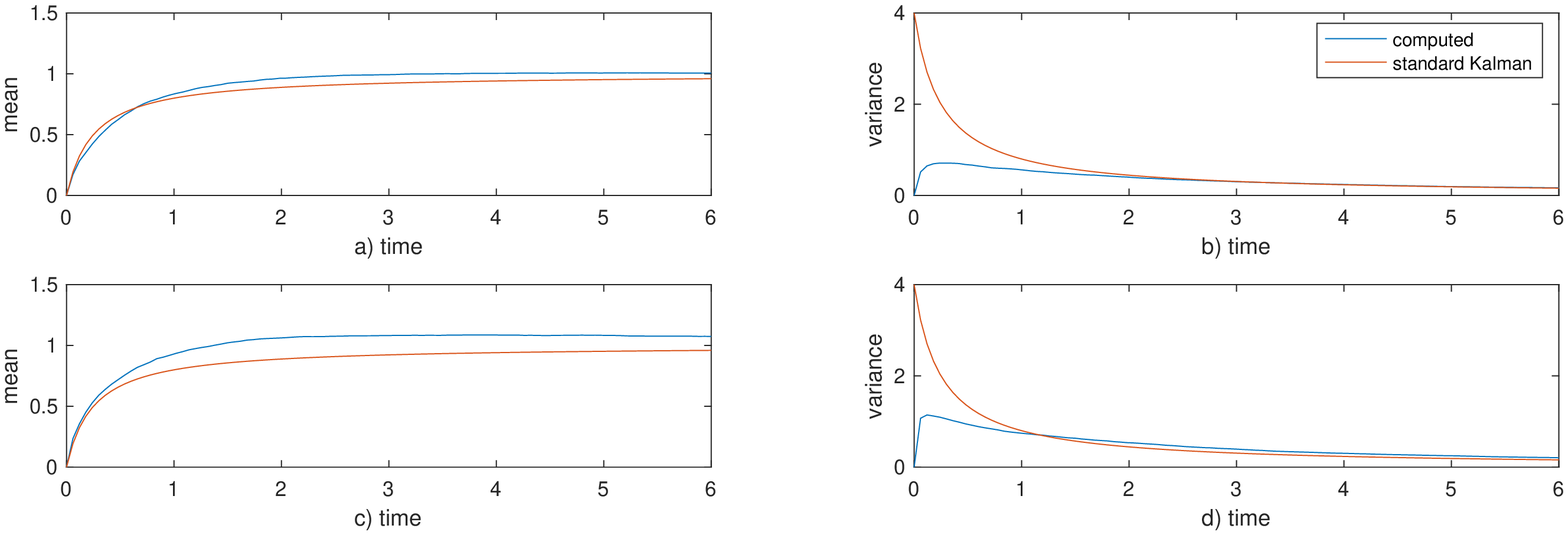}
	\end{center}
	\caption{Same experimental setting as in Figure \ref{fig2} but with the data now generated from the multi-scale SDE (\ref{eq:SDE_4}).
	Again, subsampling the data in intervals of $\Delta t =0.06$ and high-frequency assimilation with step-size $\Delta \tau = 10^{-4}$ lead to 		very similar results in terms of their frequentist means and variances.} \label{fig3}
\end{figure}

We now replace the data generating SDE model (\ref{eq:SDE_2}) by the multi-scale formulation (\ref{eq:SDE_4}) with $\epsilon = 0.01$ and $\beta = 2$. This parameter choice agrees with the one used in \cite{CNN2021}. We again find that assimilating the data at the slow time-scale $\Delta t = 0.06$ leads to very similar results obtained from an assimilation at the fast time-scale $\Delta \tau = 10^{-4}$ with the EnKBF formulation (\ref{eq:EnKF_3c}), provided the correction term resulting from the second-order iterated integral (\ref{eq:bias}) is included. See Figure \ref{fig3}. We also verified numerically that $\Delta t = 0.06$ constitutes a nearly optimal step-size in the sense of making (\ref{eq:subsampling_rate}) sufficiently small while maintaining numerical accuracy. For example, reducing the outer step-size to $\Delta t = 0.02$ leads to $h(0.02)-h(0.06) \approx 10$ in (\ref{eq:subsampling_rate}).

We finally implement the data filtering approaches (\ref{eq:EnKBF_2m}) and (\ref{eq:EnKBF_2mm}) with $\delta = 0.1$ using the true signal $X_t^\dagger$ and its multi-scale representation $X_t^{(\epsilon)}$, respectively. The numerical implementation with 
a step-size $\Delta t = \Delta \tau = 10^{-4}$ resulted in approximations $\mu_{t_n}$ which converged to the true parameter value $\theta^\dagger = 1$ as $t \to T = 6$ without the need for including further corrections terms; as expected from the results in Section \ref{sec:filtered}. 

%
\section{Conclusions} \label{sec:conclusions}
%

In this follow-up note to \cite{CNN2021}, we have investigated the impact of subsampling, data filtering, and high-frequency data assimilation on the corresponding conditional mean estimators, $\mu_t$, both for data generated from the standard SDE model and a modified multi-scale SDE. A frequentist analysis supports the basic finding that all three approaches lead to comparable results provided that the systematic biases due to different second-order iterated integrals are properly accounted for. While the EnKBF is relatively easy to analyse and a full rough path approach can be avoided, extending these results to the nonlinear feedback particle filter \cite{nusken2019state,CNN2021}  will prove more challenging. Extensions to systems without a strong scale separation \cite{Arnold2002,wouters2019stochastic} and applications to geophysical fluid dynamics \cite{Hasselmann1976,CKM11} are also of interest. In this context, the approximation quality of the proposed estimator (\ref{eq:M_estimator}) and the choice of the step-size $\Delta t$ following (\ref{eq:subsampling_rate}) (and potentially $\Delta \tau$) will be of particular interest. Finally, while we have investigated the univariate parameter estimation problem, a semi-parametric parametrisation of the drift term $f$ in (\ref{eq:SDE_1}), such as random feature maps \cite{GottwaldReich21}, lead to high-dimensional parameter estimation problems and their statistics \cite{Ghosal2017book,GineNickl2016book}. This provides another fertile direction for future research.

\paragraph{Acknowledgements.}

SR has been partially funded by Deutsche Forschungsgemeinschaft (DFG) - Project-ID 318763901 - SFB1294 and
Project-ID 235221301 - SFB1114. He would also like to thank Nikolas N\"usken for many fruitful discussions on the subject of this paper. 

%
\bibliography{refs_estimation}
%

\end{document}